\documentclass{amsart}
\usepackage{graphicx}
\usepackage{amssymb}
\usepackage{amsfonts}
\swapnumbers
\newtheorem{theorem}{Theorem}[section]

\newtheorem{corollary}[theorem]{Corollary}
\newtheorem{proposition}[theorem]{Proposition}
\theoremstyle{definition}

\newtheorem{remark}[theorem]{Remark}

\numberwithin{equation}{section}
\theoremstyle{plain}    

\numberwithin{equation}{section} %% Comment out for sequentially-numbered
\numberwithin{figure}{section} %% Comment out for sequentially-numbered
\theoremstyle{plain}    
\theoremstyle{plain}    
\theoremstyle{remark}    
\newtheorem*{acknowledgement*}{Acknowledgement} 
\newcommand{\cB}{{\mathcal B}}

\newcommand{\cF}{{\mathcal F}}

\newcommand{\cP}{{\mathcal P}}

\newcommand{\cS}{{\mathcal S}}
\newcommand{\cT}{{\mathcal T}}
\newcommand{\cU}{{\mathcal U}}

\newcommand{\Om}{{\Omega}}
\newcommand{\om}{{\omega}}
\newcommand{\ve}{{\varepsilon}}
\newcommand{\del}{{\delta}}

\newcommand{\gam}{{\gamma}}
\newcommand{\Gam}{{\Gamma}}
\newcommand{\vf}{{\varphi}}

\newcommand{\sig}{{\sigma}}
\newcommand{\al}{{\alpha}}
\newcommand{\be}{{\beta}}
\newcommand{\ka}{{\kappa}}
\newcommand{\la}{{\lambda}}
\newcommand{\La}{{\Lambda}}
\newcommand{\Up}{{\Upsilon}}
\newcommand{\up}{{\upsilon}}

\newcommand{\vs}{{\varsigma}}
\newcommand{\bbC}{{\mathbb C}}

\newcommand{\bbN}{{\mathbb N}}

\newcommand{\bbR}{{\mathbb R}}

\newcommand{\bbZ}{{\mathbb Z}}

\newcommand{\gP}{{\mathfrak P}}
\newcommand{\SRB}{{\mu^{\mbox{\tiny{SRB}}}}}

\usepackage{color}

\begin{document}
\title[]{Nonconventional large deviations theorems}%
\vskip 0.1cm 
\author{ Yuri Kifer
\quad\quad and\quad\quad\quad\quad S.R.S. Varadhan\\
\vskip 0.1cm
Institute of Mathematics\quad\quad\quad Courant Institute\\
The Hebrew University\quad\quad\quad New York University\\
Jerusalem, Israel\quad\quad\quad\quad\quad\quad New York, USA}%
\address{
Institute of Mathematics, The Hebrew University, Jerusalem 91904, Israel}
\email{ kifer@math.huji.ac.il}
\address{
Courant Institute for Mathematical Studies, New York University, 
251 Mercer St, New York, NY 10012, USA}% 
\email{ varadhan@cims.nyu.edu}%

\thanks{Yu. Kifer was supported by ISF grants 130/06 and 82/10 and S.R.S. 
Varadhan was supported by NSF grants  OISE 0730136 and DMS 0904701 }
\subjclass[2000]{Primary: 60F10 Secondary: 60J05, 60J25, 37D20}%
\keywords{large deviations, Markov processes, nonconventional averages,
hyperbolic diffeomorphisms.}%
\dedicatory{  }
\date{\today}
\begin{abstract}\noindent
We obtain large deviations theorems  for both discrete time expressions of the
form $\sum_{n=1}^NF\big(X(q_1(n)),\ldots,X(q_\ell(n))\big)$ and similar  
expressions of the form $\int_0^TF\big( X(q_1(t)),\ldots, X(q_\ell(t))\big)dt$
in continuous time. Here $X(n),n\geq 0$ or $X(t),  t\ge 0$ is a Markov process
satisfying Doeblin's condition, $F$ is a bounded continuous function 
 and $q_i(n)=in$ for $i\le k$
while for $i>k$ they are positive functions taking on integer values on
integers with some growth conditions which are satisfied, for instance, when
$q_i$'s are polynomials of increasing degrees. Applications to some types of
dynamical systems such as mixing subshifts of finite type and hyperbolic and
expanding transformations will be obtained, as well.

\end{abstract}
%\footnotetext[1]{} 
\maketitle
\markboth{Yu.Kifer and S.R.S.Varadhan}{Nonconventional large deviations} 
\renewcommand{\theequation}{\arabic{section}.\arabic{equation}}
\pagenumbering{arabic}

\renewcommand{\theequation}{\arabic{section}.\arabic{equation}}
\pagenumbering{arabic}

\section{Introduction}\label{sec1}\setcounter{equation}{0}

Nonconventional ergodic theorems which attracted substantial attention 
in ergodic theory (see, for instance, \cite{Be} and \cite{Fu}) 
studied the limits of expressions having the form
$1/N\sum_{n=1}^NT^{q_1(n)}f_1\cdots
 T^{q_\ell (n)}f_\ell$ where $T$ is a
weakly mixing measure preserving transformation, $f_i$'s are bounded 
measurable functions and $q_i$'s are polynomials taking on integer values on
the integers. While, for instance, \cite{Be} and \cite{Fu} were interested
in $L^2$ convergence, other papers such as \cite{As} provided conditions
for almost sure convergence in such ergodic theorems.
Originally, these results were motivated by applications to 
multiple recurrence for dynamical systems taking functions $f_i$ being 
indicators of some measurable sets. 

Introducing stronger mixing or weak dependence conditions enabled us in
 \cite{KV} to obtain functional central limit theorems for even more general
  expressions of the form
\begin{equation}\label{1.1}
\frac 1{\sqrt N}\sum_{n=1}^{[Nt]}\big( F(X(q_1(n)),...,X(q_\ell(n))-\bar F\big)
\end{equation}
where $X(n),\, n\geq 0$ is a sufficiently fast mixing vector valued process
with some moment conditions and stationarity properties, $F$ is a locally 
H\" older continuous function with polinomial growth, $\bar F=\int Fd(\mu\times
\cdots\times\mu)$ and $\mu$ is the distribution of $X(0)$. In order to ensure
existence of limiting variances and covariances we had to impose certain 
assumptions concerning the functions $q_j(n),\, j\geq 1$ saying that there
exists an integer $k\geq 1$ such that $q_j(n)=jn$
for $j=1,...,k$ while $q_j(n),\, j\geq k$ are positive functions taking on
integer values on integers with some (faster than linear) growth conditions.

The next natural step in the study of limiting behavior of nonconventional
sums $S_N=\sum_{n=1}^NF\big(X(q_1(n)),...,X(q_\ell(n))\big)$ is to obtain 
large deviations estimates. Namely, we will be interested in this paper
in the asymptotical behavior as $N\to\infty$ of probabilities
\begin{equation}\label{1.2}
P\{\frac 1NS_N\in\Gam\}
\end{equation}
for various (open or closed) sets $\Gam\subset\bbR$. According to \cite{Ki4}
 under appropriate
conditions $\frac 1NS_N$ converges with probability one as $N\to\infty$ to
$\bar F=\int Fd\mu\times\cdots\times\mu$ where $\mu$ is the common 
distribution of $X(n)$'s. Thus, as usual, (\ref{1.2}) describes deviations
of $\frac 1NS_N$ from the limit in the law of large numbers.

The study of asymptotics of probabilities in (\ref{1.2}) leads to what is
usually called the first level of large deviations. We will study also second
level large deviations estimates which means in our setup to consider 
occupational measures 
\begin{equation}\label{1.3}
\zeta_N=\frac 1N\sum_{n=1}^N\del_{\big(X(q_1(n)),...,X(q_\ell(n))\big)}
\end{equation}
and to study the asymptotical behavior as $N\to\infty$ of probabilities
$P\{\zeta_N\in\cU\}$
where $\cU$ is a subset in the space of probability measures on a corresponding
product space. In addition, we will consider also large deviations in the 
averaging setup, namely, for the "slow" variable $\Xi^\ve(n)=\Xi_x^\ve(n)$ 
given by a difference equation of the form
\begin{equation}\label{1.4}
\Xi^\ve(n+1)=\Xi^\ve(n)+\ve F\big(\Xi^\ve(n),X(q_1(n)),...,X(q_\ell(n))\big),
\, n=0,1,...,\,\Xi_x^\ve(0)=x
\end{equation}
which is actually a generalization of the above since if 
$F(\xi,x_1,...,x_\ell)$ does not depend on $\xi$ then $\Xi^{\frac 1N}(N)=
\frac 1NS_N$. We will deal also with continuous time versions of the above
results considering $S_T=\int_0^TF\big(X(q_1(t)),...,X(q_\ell(t))\big)dt$
for some stochastic process $X(s),\, s\geq 0$.

As for conventional sums ($\ell=k=1$) meaningful large deviations estimates
can be obtained only for some specific classes of stochastic processes 
and dynamical systems. In our more general situation we also assume that in
the probabilistic setup $X(n),\,
n=0,1,...$ is a Markov chain satisfying a (strong) Doeblin condition
while in the dynamical systems setup we can consider $X(n)=X(n,\om)=f(T^n\om)$
where $T$ is either a mixing subshift of finite type or a hyperbolic 
diffeomorphism or an expanding transformation and $f$ is a H\" older continuous
(vector) function. In the continuous time
case we take the underlying process $X(t)$ to be in the probabilistic 
setup either an irreducible finite Markov chain with continuous time
or a nondegenerate diffusion on a compact manifold while in the dynamical
systems setup we can take $X(t)=X(t,\om)=f(T^t\om)$ where $T^t,\, t\geq 0$ is
a hyperbolic flow on a compact manifold and $f$ is a H\" older continuous
(vector) function. 

We will show that it is not difficult to reduce the
problem to the case $k=\ell$ and the major problems arise only in dealing with
random variables $X(n), X(2n),...,X(kn)$. When $k=1$ the above reduction leads
to the standard (conventional) setup of large deviations. When $k>1$ then the
general case of Markov sequences requires a quite elaborate technique and a
lengthy proof and it will be treated in another paper while here when $k>1$
we restrict ourselves to independent identically distributed (i.i.d.)
sequences $X(n), n\geq 0$ which, unlike in the conventional setup, is still
nontrivial.  

Both probabilistic and dynamical
systems setups are united by common ideas and motivations but their
machineris are quite different and by this reason most of this paper deals 
with the probabilistic setup and only in the last Section \ref{sec5} we discuss
some of dynamical systems results which especially can benefit readers familiar
with this field.

\section{Preliminaries and main results}\label{sec2}\setcounter{equation}{0}

 We start with the probabilistic discrete
time setup where the underlying process $X(0),\, X(1),\, X(2),...$ is
a Markov chain defined on a probability space $(\Om,\cF,P)$ and evolving
on a Polish measurable space $(M,\cB)$ as its phase space. We assume a "strong"
 Doeblin condition saying that for some integer $n_0>0$, a constant
 $C>0$ and a probability measure $\nu$ on $M$ the $n_0$-step transition
 probability $P(n_0,x,\cdot)$ of the above Markov chain $X$ satisfies
 \begin{equation}\label{2.1}
 C^{-1}\nu(G)\leq P(n_0,x,G)\leq C\nu(G)
 \end{equation}
 for any $x\in M$ and every measurable set $G\subset M$. It is well known
 (see, for instance, \cite{Do}) that (\ref{2.1}) implies existence of 
 a unique invariant measure $\mu$ of the Markov chain $X$ and the equality
 $\mu(G)=\int d\mu(x)P(n,x,G)$ yields that 
 \begin{equation}\label{2.2}
 C^{-1}\leq\frac {d\mu}{d\nu}(x)=p(x)\leq C
 \end{equation}
 where $d\mu/d\nu$ denotes the Radon-Nikodim derivative.
 
 In all cases our setup includes also a bounded measurable function
 $F=F(x_1,x_2,...,x_\ell)$ on the $\ell$-times product space $M^\ell=
 M\times\cdots\times M$. The setup becomes complete with introduction of 
 positive increasing functions $q_j,\, j=1,...,\ell$ taking on integer values 
 on integers and such that 
 \begin{equation}\label{2.3}
 q_j(n)=jn\quad\mbox{for}\quad j=1,...,k\quad\mbox{and some}\quad k\leq\ell
 \end{equation}
 while for $j=k+1,...,\ell$ and any $\gam >0$,
 \begin{equation}\label{2.4}
 \lim_{n\to\infty}(q_j(n)-q_j(n-1))=\infty\,\,\,\mbox{and}\,\,\, 
 \liminf_{n\to\infty}(q_j(\gam n)- q_{j-1}(n))>0.
 \end{equation}
 For any function $W$ on $M^\ell$ we denote by $\hat W$ the function on $M$
 defined by
 \begin{equation}\label{2.5}
 \hat W(x)=\int\exp(W(x,x_2,...,x_\ell))d\mu(x_2)...d\mu(x_\ell).
 \end{equation}
 As usual we denote by $P_x$ the probability conditioned to $X(0)=x$ and
 by $E_x$ the corresponding expectation. Now, we can formulate our first
 result.
 
 \begin{theorem}\label{thm2.1}
 Let $W_\la(x_1,...,x_\ell),\,\la\in(-\infty,\infty)$ be a differentiable 
 in $\la$ family of bounded measurable functions on $M^\ell$ such that
 $dW_\la(x_1,...,x_\ell)/d\la$ is bounded for each $\la$, as well. Assume
 that $k=1$ in (\ref{2.3}) and (\ref{2.4}). Then for any $x\in M$ the limit
 \begin{equation}\label{2.6}
 Q(W_\la)=\lim_{N\to\infty}\frac 1N\ln E_x\exp\big(\sum_{n=1}^N
 W_\la(X(q_1(n)),...,X(q_\ell(n)))\big)
 \end{equation}
 exists, it is independent of $x$ and it is differentiable in $\la$. In fact,
 $Q(W_\la)=\ln r(W_\la)$ where $r(W)$ is the spectral radius of the positive
 operator $R(W)$ acting by 
 \begin{equation}\label{2.7}
 R(W)g(x)=\int P(x,dy)g(y)\hat W(y).
 \end{equation}
 
 Furthermore, set $W_\la(x_1,...,x_\ell)=\la F(x_1,...,x_\ell)$ and
 \begin{equation}\label{2.8}
 J(u)=\sup_\la(\la u-r(W_\la)),\, u\in\bbR.
 \end{equation}
 Then for any closed set $K\subset\bbR$,
 \begin{equation}\label{2.9}
 \limsup_{N\to\infty}\frac 1N\ln P\{\frac 1NS_N\in K\}\leq -\inf_{u\in K}
 J(u)
 \end{equation}
 and for any open set $U\subset\bbR$,
 \begin{equation}\label{2.10}
 \liminf_{N\to\infty}\frac 1N\ln P\{\frac 1NS_N\in U\}\geq -\inf_{u\in U}
 J(u)
 \end{equation}
 where, as before, $S_N=S_N(F)=\sum_{n=1}^NF\big(X(q_1(n),...,
 X(q_\ell(n))\big)$.
 \end{theorem}
 
 We observe that a very particular case of Theorem \ref{thm2.1} when $\{ X(n),
 \, n\geq 0\}$ are i.i.d. random variables was considered in Section 6 of
 \cite{Ki4-}.
 Next, we describe the second level of large deviations in the nonconventional
 setup which deals with occupational measures $\zeta_N$ on $M^\ell$ given by
 (\ref{1.3}) where $M$ is assumed to be a compact space and $\del_z$ is the
 unit mass concentrated at $z$. For any probability measure $\eta$ on $M^\ell$
 define 
 \begin{equation}\label{2.11}
 I(\eta)=-\inf_{u\in\bbC_+(M^\ell)}\int_{M^\ell}\ln\frac {E_{x_1}\int
 u(X(1),x_2,...,x_\ell)d\mu(x_2)...d\mu(x_\ell)}{u(x_1,...,x_\ell)}d\eta
 (x_1,...,x_\ell)
 \end{equation}
 where $\bbC_+(\cdot)$ denotes the space of all positive continuous functions
 on a space in brackets.
 
 \begin{theorem}\label{thm2.2} Let $k=1$ in (\ref{2.3}) and (\ref{2.4}).
 Then for any continuous function $W=W(x_1,...,x_\ell)$ on $M^\ell$ the
 limit 
 \begin{equation}\label{2.12}
 Q(W)=\lim_{N\to\infty}\frac 1N\ln E_x\exp\big(\sum_{n=1}^N
 W(X(q_1(n)),...,X(q_\ell(n)))\big)
 \end{equation}
is a convex lower semicontinuous functional satisfying
\begin{equation}\label{2.13}
Q(W)=\sup_{\eta\in\cP(M^\ell)}\big(\int W(x_1,...,x_\ell)d\eta(x_1,...,x_\ell)-
I(\eta)\big) 
\end{equation}
where $\cP(\cdot)$ is the space of probability measures on a space in brackets
considered with the topology of weak convergence.

Furthermore, for any closed set $K\subset\cP(M^\ell)$,
\begin{equation}\label{2.14}
\limsup_{N\to\infty}\frac 1N\ln P\{\zeta_N\in K\}\leq -\inf_{\eta\in K}I(\eta)
\end{equation}
and for any open set $U\subset\cP(M^\ell)$,
\begin{equation}\label{2.15}
\liminf_{N\to\infty}\frac 1N\ln P\{\zeta_N\in U\}\geq -\inf_{\eta\in U}I(\eta).
\end{equation}
\end{theorem}

Next, we exhibit continuous time versions of the above results. Here we
assume that $X(t),\, t\geq 0$ is a Markov process on a Polish measurable space
$(M,\cB)$ such that for some $t_0>0$, a constant $C>0$ and a probability
measure $\nu$ on $M$ the time $t_0$ transition probability $P(t_0,x,\cdot)$
of the above Markov process $X$ satisfies
\begin{equation}\label{2.16}
C^{-1}\nu(G)\leq P(t_0,x,G)\leq C\nu(G)
\end{equation}
for any $x\in M$ and every measurable set $G\subset M$. Again (see \cite{Do}), 
(\ref{2.16}) implies existence of a unique invariant measure $\mu$ of the
Markov process $X$ which satisfies (\ref{2.2}). Now we introduce positive
increasing functions $q_j,\, j=1,...,\ell$ on $\bbR_+$ such that for some
$0<\al_1<\al_2<...<\al_k$ and $k\leq\ell$,
\begin{equation}\label{2.17}
q_j(t)=\al_jt\quad\mbox{for}\quad j=1,...,k
\end{equation}
while for $j=k+1,...,\ell$ and any $\gam>0$,
\begin{equation}\label{2.18}
\lim_{t\to\infty}(q_j(t+\gam)-q_j(t))=\infty\,\,\,{and}\,\,\,
\liminf_{t\to\infty}(q_j(\gam t)-q_{j-1}(t))>0.
\end{equation}
We will be interested in large deviations estimates as $T\to\infty$ for
\[
S_T(F)=S_T=\int_0^TF\big(X(q_1(t)),...,X(q_\ell(t))\big)dt.
\]

\begin{theorem}\label{thm2.3}
 Let $W_\la(x_1,...,x_\ell),\,\la\in(-\infty,\infty)$ be as in 
 Theorem \ref{thm2.1}. Assume that $k=1$ in (\ref{2.17}) and (\ref{2.18}). 
 Then for any $x\in M$ the limit
 \begin{equation}\label{2.19}
 Q_{\mbox{cont}}(W_\la)=\lim_{T\to\infty}\frac 1T\ln E_x\exp\big(\int_0^T
 W_\la(X(q_1(t)),...,X(q_\ell(t)))dt\big)
 \end{equation}
 exists, it is independent of $x$ and it is differentiable in $\la$. In fact,
 $Q_{\mbox{cont}}(W_\la)=\ln r_{\mbox{cont}}(W_\la)$ where $r_{\mbox{cont}}(W)$
 is the spectral radius of the semigroup of positive operators 
 $R^t_{\mbox{cont}}(W)$ acting by the formula
 \begin{equation}\label{2.20}
 R^t_{\mbox{cont}}(W)g(x)=E_x\big(g(X(t))\hat W_{\mbox{cont}}(t)\big)
 \end{equation}
 where
 \begin{equation}\label{2.21}
 \hat W_{\mbox{cont}}(t)=\exp\big(\int_0^tds\int W_\la(X(\al_1 s),x_2,...,
 x_\ell) d\mu(x_2)...d\mu(x_\ell)\big).
 \end{equation}
 
 Furthermore, set $W_\la(x_1,...,x_\ell)=\la F(x_1,...,x_\ell)$ and
 define $J(u)=J_{\mbox{cont}}(u)$ by (\ref{2.8}) with $r_{\mbox{cont}}$
 in place of $r$. Then for any closed set $K\subset\bbR$,
 \begin{equation}\label{2.22}
 \limsup_{T\to\infty}\frac 1T\ln P\{\frac 1TS_T\in K\}\leq -\inf_{u\in K}
 J(u)
 \end{equation}
 and for any open set $U\subset\bbR$,
 \begin{equation}\label{2.23}
 \liminf_{T\to\infty}\frac 1T\ln P\{\frac 1TS_T\in U\}\geq -\inf_{u\in U}
 J(u).
 \end{equation}
 \end{theorem}
 
The second level of large deviations in the continuous time nonconventional
setup deals with occupational measures
\begin{equation}\label{2.24}
\zeta_T=\frac 1T\int_0^T\del_{\big(X(q_1(t)),...,X(q_\ell(t))\big)}dt
\end{equation}
on $M^\ell$. Now we assume that $X(t),\, t\geq 0$ is a diffusion process on
a compact Riemannian manifold $M$ with the generator $L$ which is a
nondegenerate second order elliptic differential operator. For any probability
measure $\eta$ on $M^\ell$ set
\begin{equation}\label{2.25}
I_{\mbox{cont}}(\eta)=-\inf_{u\in D_+}\int_M\frac {L_{x_1}u(x_1,x_2,...,x_\ell)
d\mu(x_2)...d\mu(x_\ell)}{u(x_1,x_2,...,x_\ell)}d\eta(x_1,...,x_\ell)
\end{equation}
where the infimum is taken over all positive $u$ from the domain of $L$.

\begin{theorem}\label{thm2.4} Let $k=1$ in (\ref{2.17}) and (\ref{2.18}).
 Then for any continuous function $W=W(x_1,...,x_\ell)$ on $M^\ell$ the limit 
 \begin{equation}\label{2.26}
 Q_{\mbox{cont}}(W)=\lim_{T\to\infty}\frac 1T\ln E_x\exp\big(\int_0^T
 W(X(q_1(t)),...,X(q_\ell(t)))dt\big)=r_{\mbox{cont}}(W) 
 \end{equation}
is a convex lower semicontinuous functional satisfying
\begin{equation}\label{2.27}
Q_{\mbox{cont}}(W)=\sup_{\eta\in\cP(M^\ell)}\big(\int W(x_1,...,x_\ell)
d\eta(x_1,...,x_\ell)-I_{\mbox{cont}}(\eta)\big). 
\end{equation}

Furthermore, for any closed set $K\subset\cP(M^\ell)$,
\begin{equation}\label{2.28}
\limsup_{T\to\infty}\frac 1T\ln P\{\zeta_T\in K\}\leq -\inf_{\eta\in K}
I_{\mbox{cont}}(\eta)
\end{equation}
and for any open set $\subset\cP(M^\ell)$,
\begin{equation}\label{2.29}
\liminf_{T\to\infty}\frac 1T\ln P\{\zeta_T\in U\}\geq -\inf_{\eta\in U}
I_{\mbox{cont}}(\eta).
\end{equation}
\end{theorem}
A similar result holds true when $X(t)$ is a nondegenerate continuous
time Markov chain with a finite state space.

Next, we describe our large deviations estimates in a nonconventional
averaging setup. Here we consider either a difference equation (\ref{1.4})
for $\Xi^\ve(n)$ in the discrete time case where $X(n),\, n\geq 0$ is a
Markov chain satisfying conditions of Theorem \ref{thm2.1} or a differential
equation for $\Xi^\ve(t)=\Xi_x^\ve(t)\in\bbR^d$, $t\geq 0$,
\begin{equation}\label{2.30}
\frac {d\Xi^\ve(t)}{dt}=\ve F\big(\Xi^\ve(t),X(q_1(t)),...,X(q_\ell(t))\big),
\, \Xi_x^\ve(0)=x
\end{equation}
in the continuous time setup where $X(t),\, t\geq 0$ is a Markov process
satisfying conditions of Theorem \ref{2.3}. We assume that $F(\xi,x_1,...,
x_\ell)$ is bounded and Lipschitz continuous in $\xi$. The setup of 
(\ref{2.30}) emerges considering, for instance, a time dependent small
perturbation of the oscillator equation
\begin{equation}\label{2.31}
 \ddot x+\la^2x=\ve g(x,\dot x,t)
\end{equation}
where the force term $g$ depends on time in a random way $g(x,y,t)=g(x,y,
X(q_1(t)),...,X(q_\ell(t)))$. Then passing to the polar coordinates 
$(r,\phi)$ with $x=r\sin(\la(t-\phi))$ and $\dot x=\la r\cos(\la(t-\phi))$ 
the equation (\ref{2.31}) will be transformed into (\ref{2.30}) with $\Xi^\ve
=(r,\phi)$. It seems 
reasonable that a random force may depend on versions of a same process 
 moving with different speeds which is what we have here.
 
As it is well known (see, for instance, \cite{SVM}), if $F(\xi,x_1,...,
x_\ell)$ is bounded and Lipschitz continuous in $\xi$ then whenever for each
$\xi$ the (pointwise) limit
\begin{equation*}
\bar F(\xi)=\lim_{\cT\to\infty}\frac 1\cT\int_0^\cT F(\xi,X(q_1(t)),...,
X(q_\ell(t)))dt
\end{equation*}
exists then for any $T\geq 0$,
\begin{equation*}
\lim_{\ve\to 0}\sup_{0\leq t\leq T/\ve}|\Xi^\ve(t)-\bar \Xi^\ve(t)|=0
\end{equation*}
where
\begin{equation*}
\frac {d\bar \Xi^\ve(t)}{dt}=\ve\bar F(\bar \Xi^\ve(t)).
\end{equation*}
In the discrete time case we have to take
\begin{equation*}
\bar F(\xi)=\lim_{N\to\infty}\frac 1N\sum_{n=0}^N F(\xi,X(q_1(n)),...,
X(q_\ell(n))).
\end{equation*}
 Almost everywhere limits of the averages above can
 be obtained by nonconventional pointwise ergodic theorems from \cite{BLM} 
 and \cite{As}, respectively, in rather general circumstances in the dynamical
 systems case and under another set of conditions existence of such limits 
 follows from \cite{Ki4}. The next natural step here is to obtain large
 deviations estimates for the above approximation of the slow motion $\Xi^\ve$
 by the averaged one $\bar\Xi^\ve$.

For any $\eta\in\cP(M^\ell)$ set 
\begin{equation}\label{2.32}
\bar B_\eta(\xi)=\int B(\xi,x_1,...,x_\ell)d\eta(x_1,...,x_\ell).
\end{equation}
For each absolutely continuous curve $\gam_t,\, t\in[0,\cT]$ set
\begin{equation*}
\cS_{0T}(\gam)=\int_0^\cT\inf\{ I(\eta):\,\dot\gam_t =\bar B_\eta(\gam_t)\}dt
\end{equation*}
where $I(\eta)$ is given by (\ref{2.11}) or $I(\eta)=I_{\mbox{cont}}(\eta)$
given by (\ref{2.25}) in the discrete or continuous time cases, respectively.
If $\gam_t,\, t\in[0,T]$ is not absolutely continuous we set $S_{0T}(\gam)
=\infty$.

\begin{theorem}\label{thm2.5} Let $k=1$ in (\ref{2.3}) and (\ref{2.4}) or in
(\ref{2.17}) and (\ref{2.18}) and set
$\Psi^\ve(t)=\Xi^\ve([t/\ve])$ or $\Psi^\ve(t)=\Xi^\ve(t/\ve)$ in the 
discrete or continuous time cases, respectively. Then for any continuous 
function $W_t(x_1,...,x_\ell)$ on $\bbR_+\times M^\ell$,
\begin{equation}\label{2.33}
\lim_{\ve\to 0}\ve\ln E_x\exp\big(\ve^{-1}\int_0^TW_t\big( X(q_1(t/\ve)),...,
X(q_\ell(t/\ve))\big)dt\big)=\int_0^Tr_{\mbox{cont}}(W_t)dt
\end{equation}
where $r_{\mbox{cont}}$ is the same as in Theorem \ref{thm2.3} with $W_t$ 
considered as a function on $M^\ell$ and in the discrete time case we either 
extend $q_j(t)=q_j([t])$ to all $t\geq 0$ in order to write the integral in 
exponent in (\ref{2.33}) or replace this integral by the corresponding sum.

Furthermore, for any $a,\del,\la>0$ 
and every continuous $\gam_t,\, t\in[0,\cT]$, $\gam_0=x$ there exist $\ve_0>0$
such that for all positive $\ve<\ve_0$,
\begin{equation}\label{2.34}
P\{\,\rho_{0,\cT}(\Psi_x^\ve,\gam)<\del\}\geq\exp\{
-\frac 1\ve(\cS_{0,\cT}(\gam)+\la)\}\quad\mbox{and}
\end{equation}
\begin{equation}\label{2.35}
P\{\,\rho_{0,\cT}(\Psi_x^\ve,\Phi^a_{0,\cT}(x))\geq\del\}\leq
\exp\{-\frac 1\ve(a-\la)\}
\end{equation}
where $\Psi_x^\ve(0)=x$, $\rho_{0,\cT}$ is the uniform distance and 
$\Phi_{0,\cT}^a(x)=\{\gam:\,\gam_0=x,\, \cS_{0,T}(\gam)\leq a\}$.
\end{theorem}

\begin{remark}\label{rem2.6}
Suppose that the averaged motion $\bar \Xi^\ve$ has several attracting fixed
 points and limit circles. Then similarly to \cite{Fr} (Markov chains
 case) and \cite{Ki3} (dynamical systems case) we can study rare transitions
 of the slow motion $\Xi^\ve$ between these attractors. However, in the 
 nonconventional setup the situation is more complicated and this problem
 will not be dealt with in this paper.
 \end{remark}

Certain versions of Theorems \ref{thm2.2}--\ref{thm2.5} can be obtained for 
some classes of dynamical systems such as mixing subshifts of finite type and
$C^2$ hyperbolic and expanding transformations but in order not to interrupt
 probabilistic exposition here we discuss some of these results in the last
Section \ref{sec5}.

In the next section we will show that the study of large deviations in our
nonconventional setup can be always reduced to the case $k=\ell$, i.e. we have 
to deal only with $q_j(n)=jn,\, j=1,...,k$. So we discuss next this situation
allowing any $k\geq 1$ while assuming that $X(n),\, n\geq 0$, $q_j$ and $F$
are the same as in Theorem \ref{thm2.1}. It turns out that the treatment of
the general case when $X(0),X(1),X(2),...$ is a Markov chain requires a quite
 complicated and technical proof whose exposition here would make this paper
 too long, and so it will be discussed in another paper. Thus, we will restrict
ourselves here to a particular case when $X(n),\, n\geq 0$ are independent 
identically distributed (i.i.d.) random variables (or vectors).
 Namely, we are interested in large deviations estimates for $S_N(F)=
\sum_{n=1}^NF(X(n),X(2n),...,X(kn))$ where $X(n)\in M,\, n\geq 1$ are i.i.d.
random variables (vectors) with a compact support $M$. Let 
$r_1,...,r_m\geq 2$ be all primes not exceeding $k$.
Set $A_n=\{ a\leq n:\, a\,\,\mbox{is relatively prime with}\, r_1,...,r_m\}$ 
and $B_\eta(a)=\{ b\leq \eta:\, b=ar_1^{d_1}r_2^{d_2}\cdots r_m^{d_m}$ for some
nonnegative integers $d_1,...,d_m\}$. Now for any bounded measurable 
function $V$ on $M^k$ we write
\begin{equation}\label{2.36}
S_N(V)=\sum_{a\in A_N}S_{N,a}(V)\,\,\mbox{with}\,\, S_{N,a}(V)=
\sum_{b\in B_N(a)}V(X(b),X(2b),...,X(kb)).
\end{equation}
Observe that $S_{N,a}(V),\, a\in A_V$ is a collection of independent random
variables.

\begin{theorem}\label{thm2.6}
For any continuous function $V$ on $M^k$ the limit
\begin{eqnarray}\label{2.37}
&Q(V)=\lim_{N\to\infty}\frac 1N\ln E\exp\big(\sum_{n=1}^NV(X(n),X(2n),...
,X(kn))\big)\\
&=\lim_{N\to\infty}\frac 1N\sum_{a\in A_N}\ln E\exp S_{N,a}(V)\nonumber
\end{eqnarray}
exists and the functional $Q(V)$ is convex and lower semicontinuous.
If $V=V_\la$ depends on a parameter $\la$ and has a bounded derivative in
$\la$ then 
$Q(V_\la)$ is also differentiable in $\la$. Thus taking $V_\la=\la F$ we
obtain that also for $k\geq 2$ in the above i.i.d. setup both upper and lower
large deviations bounds (\ref{2.9}) and (\ref{2.10}) hold true with the rate 
functional $J$ being the Fenchel-Legendre transform $J(u)=\sup_\la(\la u-
Q(\la F))$ of $Q$. 
\end{theorem}

In Section \ref{sec4} we will provide a rather explicit computation of the
limit (\ref{2.37}). As a model application of Theorem \ref{thm2.6} we can 
consider digits $X(n)=X(n,\om),\, n\geq 1$ of base $M$ expansions 
$\om=\sum_{n=1}^\infty \frac{X(n,\om)}{M^n}$, $X(n,\om)\in\{ 0,1,...,M-1\}$ 
of numbers $\om\in[0,1)$ which are i.i.d. random variables on the probability
space $([0,1),\cB,P)$ where $\cB$ is the Borel $\sig$-algebra and $P$ is the
Lebesgue measure. Take, for instance, $V(x_1,...,x_k)=\del_{\al_1x_1}
\del_{\al_2x_2}\cdots\del_{\al_kx_k}$ for some $\al_1,...,\al_k
\in\{ 0,1,...,M-1\}$ with $\del_{ij}=1$ if $i=j$ and $=0$, otherwise.
Then Theorem \ref{thm2.6} provides large deviations estimates for the number
\begin{eqnarray}\label{2.38}
&n_{\al_1,...,\al_k}(N,\om)=
\#\{ n\leq N:\, X(n,\om)=\al_1,X(2n,\om)=\al_2,\\
&...,X(kn,\om)=\al_k\}=\sum_{n=1}^NV(X(n,\om),...,X(kn,\om)).\nonumber
\end{eqnarray}

The same setup can be reformulated in the following way. Consider infinite
sequences of letters (colors, spins, etc.) taken out of an alphabet of
size $M$. Let $n_{\al_1,...,\al_k}(N)$ be the number of arithmetic 
progressions of length $k$ with both the first term and the difference equal
$n\leq N$ and having the letter (color, spin, etc.)
$\al_i$ on the place $i=1,2,...,k$. Then Theorem \ref{thm2.6}
yields large deviations bounds for $n_{\al_1,...,\al_k}(N)$ as $N\to\infty$
considered as a random variable on the space of sequences of letters with
any product probability measure, in particular, with uniform probability
measure which assigns the same weight to each combination of
$n$ consecutive letters (i.e. to each cylinder set of length $n$) for all
$n=1,2,...$.
We observe that another statistical physics interpretation of a particular
case of the above i.i.d. setup appeared independently in a recent paper 
\cite{CCGR} though large deviations bounds were obtained there only 
for the case $k=M=2$.

\section{Large deviations for Markov processes: $k=1$ case}
\label{sec3}\setcounter{equation}{0}

\subsection{Reduction to the $k=\ell$ case}\label{subsec3.1}

First, we will show that the study of the limit (\ref{2.6}) for any
$k\leq\ell$ can be reduced to the case $k=\ell$. In order to apply this
result not only to Markov chains but also to other fast mixing processes,
in particular to dynamical systems considered in Section \ref{sec5}, we will 
deal here with a somewhat more general setup.

Let $\{ X(n),\, n=0,1,...\}$ be a sequence of measurable mappings of a
measurable space $(\Om,\cF)$ to a Polish space $M$ considered
with its Borel $\sig$-algebra $\cB$. Since $(M,\cB)$ is isomorphic to a Borel
subset $\Up$ of an interval we can and do identify $M$ with $\Up$ and assume
that each $X(n)$ is real (or vector) valued. Then $\{ X(n),\, n=0,1,...\}$
becomes a real (or vector) valued stochastic process under each 
probability measure on $(\Om,\cF)$. Our setup includes two such measures $P$
and $\Pi$ while we assume that $X(n)\Pi=\mu$ does not depend on $n$, i.e. that
the one dimensional distribution $\mu$ of $X(n)$ on the probability space 
$(\Om,\cF,\Pi)$ is the same for all $n$. In order to state our conditions we
introduce also a family of 
$\sig$-algebras $\cF_{ml}\subset\cF$, $-\infty\leq m\leq l\leq\infty$
satisfying $\cF_{-\infty,\infty}=\cF$ and $\cF_{ml}\subset\cF_{m'l'}$ if 
$m'\leq m$ and $l'\geq l$. Next, we define a modified
$\psi$-mixing (dependence) coefficient by
\begin{eqnarray*}
&\psi(n)=\psi_{P,\Pi}(n)=\sup_{l\geq 0,g}\{\| E_P(g|\cF_{-\infty,l})-
E_\Pi g\|_\infty :\\
&\,g\,\,\mbox{is}\,\,\cF_{l+n,\infty}-\mbox{measurable and}\,\,
E_\Pi|g|\leq 1\}
\end{eqnarray*}
where $E_Q$ is the expectation with respect to a probability measure $Q$
and $\|\cdot\|_\infty$ is the $L^\infty(\Om,P)$ norm.
The rational behind introduction of two probability measures $P$ and $\Pi$ 
above is to allow $X(n),\, n\geq 0$ to be a Markov chain with an arbitrary
initial distribution (in particular, starting at a point) under $P$ while
$X(n)$ is stationary under $\Pi$ and the distribution of $X(n)$ under $P$
converges to $\mu=X(0)\Pi$. Furthermore, we will not assume measurability
of $X(n)$'s with respect to some of $\sig$-algebras $\cF_{m,l}$ but instead
will rely on approximation coefficients defined for each bounded continuous
function $V=V(x_1,...,x_\ell)$ on $M^\ell$ by
\begin{eqnarray*}
&\be_V(n)=\max_{1\leq j\leq\ell}\sup_{x_1,...,x_{j-1},x_{j+1},...,x_\ell\in M}
\sup_{m\geq 0}\|V(x_1,...,x_{j-1},\\
&X(m),x_{j+1},...,x_\ell)-V(x_1,...,x_{j-1},E_P(X(m)|\cF_{m-n,m+n}),
x_{j+1},...,x_\ell)\|_\infty.
\end{eqnarray*}
Since $V$ is continuous we can take here the supremum over a countable dense
set in $M^{\ell-1}$, and so outside of one $P$-measure zero set $\be_V(n)$
gives a uniform bound of the difference above.

\begin{proposition}\label{prop3.1}
Let $V(x_1,...,x_\ell)$ be a bounded continuous function on $M^\ell$ and 
assume that
\begin{equation}\label{3.1}
\lim_{n\to\infty}(\psi(n)+\be_V(n))=0
\end{equation}
together with the conditions (\ref{2.3}) and (\ref{2.4}) on functions $q_j,\, 
j=1,...,\ell$. Then,
\begin{eqnarray}\label{3.2}
&\lim_{N\to\infty}\frac 1N\big(\ln E_P\exp\big(\sum_{n=1}^NV(X(q_1(n)),...,
X(q_\ell(n)))\big)\\
&-\ln E_P\exp\big(\sum_{n=1}^NV^{(k)}(X(n),X(2n),...,X(kn)))\big)\big)=0
\nonumber\end{eqnarray}
where for each $m<\ell$,
\begin{eqnarray}\label{3.3}
&V^{(m)}(x_1,...,x_m)=\ln\int_M...\int_M\exp(V(x_1,...,x_m,x_{m+1},...,x_\ell))
\\
&d\mu(x_{m+1})...d\mu(x_\ell)\,\,\mbox{and}\,\, V^{(\ell)}=V.\nonumber
\end{eqnarray}
If, in fact, $X(n)$ is $\cF_{n,n}$-measurable then (\ref{3.2}) holds true
for any bounded measurable function $V$ assuming only that $\psi(n)\to 0$
as $n\to\infty$.
\end{proposition}
\begin{proof} 
Observe that (\ref{2.4}) yields
\begin{equation}\label{3.4}
\lim_{n\to\infty}(q_j(\gam n)-q_{j-1}(n))=\infty\,\,\mbox{for any}\,\,
j>k\,\,\mbox{and}\,\,\gam>0.
\end{equation}
Set
\[
d_\gam(n)=\min_{k+1\leq j\leq\ell}\min\big(q_j(\gam n)-q_{j-1}(n),
\min_{l\geq\gam n}(q_j(l)-q_j(l-1))\big)
\]
and observe that $d_\gam(n)\to\infty$ as $n\to\infty$ in view of (\ref{2.4})
and (\ref{3.4}). For any $l=0,1,...$ and $0\leq r\leq\infty$ set
\[
X_r(l)=E_P(X(l)|\cF_{l-r,l+r}).
\]
Next, for $m=1,2,...,\ell$, $a\leq b\leq c$ and $0\leq r\leq\infty$ denote
\begin{eqnarray*}
&Z_r^{(m)}(a,b,c)=E_P\exp\big(\sum_{a<l\leq b}V^{(m)}(X_r(q_1(l)),...,
X_r(q_m(l)))\\
&+\sum_{b<l\leq c}V^{(m-1)}(X_r(q_1(l)),...,X_r(q_{m-1}(l)))\big).
\end{eqnarray*}
If $b=c$, i.e. we have only the first sum above, we set $Z_r^{(m)}(a,b,c)=
Z_r^{(m)}(a,b)$. If $r=\infty$ we drop the index $r$ and write just
$Z^{(n)}(a,b,c)$ or $Z^{(m)}(a,b)$. Observe that
\begin{equation}\label{3.5}
e^{-C(V)\gam N}Z_r^{(m)}(\gam N,N)\leq Z^{(m)}_r(0,N)\leq e^{C(V)\gam N}
Z_r^{(m)}(\gam N,N)
\end{equation}
where $C(V)=\sup_{(x_1,...,x_\ell)\in M^\ell}|V(x_1,...,x_\ell)|$. By the
definition of $\be_V(n)$ (and the remark after it) we obtain also that for
any $m=1,2,...,\ell$, $a\leq b\leq c$ and $0\leq r\leq\infty$,
\begin{equation}\label{3.6}
Z_r^{(m)}(a,b,c)e^{-(c-a)\ell\be_V(r)}\leq Z^{(m)}(a,b,c)\leq Z_r^{(m)}(a,b,c)
e^{(c-a)\ell\be_V(r)}.
\end{equation}

Let $g=g(x,y)$ be a bounded measurable function on a product $M_1\times M_2$
(for some measurable spaces $(M_1,\cB_1)$ and $(M_2,\cB_2)$) and
$X:\Om\to M_1$ and $Y:\Om\to M_2$ be $\cF_{-\infty,l}-$ and 
$\cF_{l+n,\infty}-$measurable random variables (maps), respectively. Then it
follows from the definition of $\psi(n)=\psi_{P,\Pi}(n)$ that
\begin{equation}\label{3.7}
|E_P(g(X,Y)|\cF_{-\infty,l})-g_\Pi(X)|\leq \psi(n)|g|_\Pi(X)
\end{equation}
where $g_\Pi(x)=E_\Pi g(x,Y)$ and $|g|_\Pi(x)=E_\Pi |g(x,Y)|$. Now take
$r=r_\gam(N)=[\frac 13d_\gam(N)]$ where $[\cdot ]$ denotes the integral part.
 Then for all $N\geq n\geq\gam N+1$, $m=k+1,...,\ell$ and $N$ large enough
\begin{eqnarray}\label{3.8}
&Z_r^{(m)}(\gam N,n,N)=E_P\big( J\exp\big(\sum_{\gam N<l\leq n-1}
V^{(m)}(X_r(q_1(l)),...,X_r(q_m(l)))\\
&+\sum_{n<l\leq N}V^{(m-1)}(X_r(q_1(l)),...,X_r(q_{m-1}(l)))\big)\nonumber
\end{eqnarray}
where
\begin{equation*}
J=J_r(n)=E_P\big(\exp\big(V^{(m)}(X_r(q_1(n)),...,X_r(q_m(n)))\big)\big\vert
\cF_{-\infty,q_m(n-1)+r}\big).
\end{equation*}
By (\ref{3.7}) and the definition of $\be_V$ we conclude that
\begin{eqnarray}\label{3.9}
&\big\vert J-\int\exp\big(V^{(m)}(X_r(q_1(n)),...,X_r(q_{m-1}(n)),y)\big)
d\mu(y)\big\vert\\
&\leq\eta(r)\int\exp\big(V^{(m)}(X_r(q_1(n)),...,X_r(q_{m-1}(n)),y)\big)d\mu(y)
\nonumber\end{eqnarray}
where $\eta(n)=(\psi(n)+2\be_V(n)+2\be_V(n)\psi(n))e^{C(V)}\to 0$ as 
$n\to\infty$.
Employing (\ref{3.8}) and (\ref{3.9}) for $n=N,N-1,...,[\gam N]+1$ we obtain
 that
 \begin{equation}\label{3.10}
 (1-\eta(r))^NZ_r^{(m-1)}(\gam N,N)\leq Z_r^{(m)}(\gam N,N)\leq 
 (1+\eta(r))^NZ_r^{(m-1)}(\gam N,N).
 \end{equation}
 Next, we use (\ref{3.10}) for $m=\ell,\ell-1,...,k+1$ which together
 with (\ref{3.5}) and (\ref{3.6}) yields that
 \begin{eqnarray}\label{3.11}
 &(1-\eta(r))^{\ell N}e^{-2N(\ell\be_V(r)+C(V)\gam)}Z^{(k)}(0,N)\leq 
 Z^{(\ell)}(0,N)\\
 &\leq (1+\eta(r))^{\ell N}e^{2N(\ell\be_V(r)+C(V)\gam)}Z^{(k)}(0,N).\nonumber
 \end{eqnarray}
 Taking $\ln$ in (\ref{3.11}), dividing by $N$, letting $N\to\infty$ and
 taking into account that then $r=r(N)\to\infty$, we obtain that
 \begin{equation*}
 \limsup_{N\to\infty}\frac 1N\big\vert\ln Z^{(\ell)}(0,N)-\ln Z^{(k)}(0,N)
 \big\vert\leq 2C(V)\gam
 \end{equation*}
 and (\ref{3.2}) follows since $\gam>0$ is arbitrary.
 
 If $X(n)$ is $\cF_{n,n}$-measurable for each $n$ then we do not have to
 deal with the approximation coefficient $\be_V(r)$ and $X_r=X,\, Z_r^{(m)}=
 Z^{(m)}$ above. Hence all above arguments remain true with $\be_V(r)=0$
 for any bounded measurable $V$ and we obtain (\ref{3.2}) provided 
 $\psi(n)\to 0$ as $n\to\infty$.
\end{proof}

It is easy to check the conditions of Proposition \ref{prop3.1} for Markov
chains $X(n),\, n\geq 0$ satisfying the "strong" Doeblin condition (\ref{2.1}).
 Indeed,
denote by $\cF_{l,m},\, l\leq m$ the $\sig$-algebra generated by $X(l),...,
X(m)$ with $\cF_{l,\infty}$ being the minimal $\sig$-algebra containing all
$\cF_{l,m},\, m\geq l$ and we set $\cF_{l,m}=\cF_{0,m}$ for $l<0$ and $m\geq 0$.
If $g$ is $\cF_{l+n,\infty}-$measurable then by the Markov property
\begin{equation}\label{3.12}
E_P(g|\cF_{-\infty,l})=\int P(n,X(l),dy)E_{P_y}g
\end{equation}
where $P_y$ is the probability measure on the path space of the Markov 
chain $X(n)$ starting at $y$. The Chapman-Kolmogorov equation sais that for 
any $n\geq n_0$,
\[
P(n,x,G)=\int P(n-n_0,x,dy)P(n_0,y,G),
\]
and so by (\ref{2.1}) for all such $n$,
\begin{equation*}
C^{-1}\nu(G)\leq P(n,x,G)\leq C\nu(G).
\end{equation*}
This together with the Radon-Nikodim theorem yields existence for 
$\nu$-almost all $y$ and $n\geq n_0$ of the transition density $p(n,x,y)$
satisfying
\begin{equation*}
C^{-1}\leq p(n,x,y)=\frac {dP(n,x,\cdot)}{d\nu}(y)\leq C.
\end{equation*}
It is well known (see, for instance, \cite{Do}) that (\ref{2.1}) and
 (\ref{2.2}) imply that
\begin{equation}\label{3.13}
(1-Ke^{-\ka n})p(y)\leq p(n,x,y)\leq (1+Ke^{-\ka n})p(y)
\end{equation}
for some $K,\ka>0$ independent of $n\geq n_0$. If $\Pi$ is the stationary
probability of the Markov chain on the path space then
\[
E_\Pi g=\int p(y)E_{P_y}gd\nu(y).
\]
Hence, by (\ref{3.12}) and (\ref{3.13}),
\[
\| E_P(g|\cF_{-\infty,l})-E_\Pi g\|_\infty\leq Ke^{-\ka n}E_\Pi|g|.
\]
Thus the condition (\ref{3.1}) with $\be_V(n)=0$ is satisfied in our 
Markov chains case.

\begin{corollary}\label{cor3.2}
Assume that conditions of Proposition 3.1 hold true. Suppose that for any
bounded measurable function $V_\la(x_1,...,x_k)$ on $\bbR\times M^k$ having
a bounded in $x_1,...,x_k$ derivative in a parameter $\la\in(-\infty,\infty)$
the limit
\begin{equation*}
Q(V_\la)=\lim_{N\to\infty}\frac 1N\ln E_x\exp\big(\sum_{n=1}^NV_\la(X(n),X(2n),
...,X(kn))\big)
\end{equation*}
exists, it is a lower semicontinuous convex functional and it is differentiable
in the parameter $\la$. Then for any
bounded measurable function $W_\la(x_1,...,x_\ell)$ on $\bbR\times M^\ell$ 
having a bounded in $x_1,...,x_\ell$ derivative in a parameter 
$\la\in(-\infty,\infty)$ the limit
\begin{equation*}
Q(W_\la)=\lim_{N\to\infty}\frac 1N\ln E_x\exp\big(\sum_{n=1}^NW_\la(X(q_1(n)),
,...,X(q_\ell(n)))\big)=Q(W_\la^{(k)})
\end{equation*}
exists, it is a lower semicontinuous convex functional and it is differentiable
in the parameter $\la$. In particular, the large deviations estimates in the
form (\ref{2.9}) and (\ref{2.10}) hold true then with the rate functional
$J$ given by (\ref{2.8}) with $W_\la=\la F$.
\end{corollary}
\begin{proof} By Proposition \ref{prop3.1}, $Q(W_\la)=Q(W_\la^{(k)})$ and 
we see from (\ref{3.3}) that if $W_\la$ is bounded and has a bounded derivative
in $\la$ then so does $W_\la^{(k)}$. Hence, by the assumption $Q(W_\la^{(k)})$
is a lower semicontinuous convex functional and it is differentiable in $\la$
which implies the same for $Q(W_\la)$ and the result follows. 
\end{proof}

Now let $k=1$ and $V=W_\la$ as in Theorem \ref{thm2.1}. Then $\hat {W}_\la=
V^{(1)}$ and by Proposition \ref{3.1}, 
\begin{equation}\label{3.14}
Q(W_\la)=\lim_{N\to\infty}\frac 1N\ln E_x\exp\big(\sum_{n=1}^N
\hat {W}_\la(X(n))\big).
\end{equation}
Thus we arrive at the standard limit appearing in "conventional" large 
deviations results which is well known for Markov chains $X(n),\, n\geq 0$
satisfying our conditions as it is described in Theorem \ref{thm2.1}.
Differentiability of $Q(W_\la)$ in $\la$ follows from standard results on 
positive operatos (see, for instance, \cite{Kr}) and we derive now
Theorem \ref{thm2.1} from well known "conventional" large deviations results
(see, for instance, \cite{DV1}, \cite{Ki2} and Section 2.3 in \cite{DZ}). \qed

\subsection{2nd level of large deviations}\label{subsec3.2}

Recall that in the setup of Theorem \ref{thm2.2} we have $k=1$, $M$ being
a compact space and the result is about large deviations for occupational
measures $\zeta_N$ appearing there. Let $W$ be a continuous function on
$M^\ell$ with $\hat W$ defined by (\ref{2.5}). By Proposition \ref{prop3.1}
together with the well known facts (see, for instance, \cite{DV1}, \cite{DZ}
and \cite{Ki2}),
\begin{equation}\label{3.15}
Q(W)=\lim_{N\to\infty}\frac 1N\ln E_x\exp\big(\sum_{n=1}^NW(X(q_1(n)),...,
X(q_\ell(n)))\big)=\ln(r(W))
\end{equation}
where $r(W)$ is the spectral radius of the operator
\begin{equation}\label{3.16}
R(W)g(x)=E_x\big(g(X(1))\hat W(X(1))\big)=E_x\big(g(X(1))e^{\ln\hat W(X(1))}\big).
\end{equation}
 Observe, that by the Donsker--Varadhan variational formula (see \cite{DV1} and
 \cite{DV2}),
\begin{equation}\label{3.17}
Q(W)=\sup_{\nu\in\mathcal P(M)}(\int_M\ln\hat W(x)d\nu(x)-\hat I(\nu))
\end{equation}
where $\hat I(\nu)=-\inf_{u\in C_+(M)}\int\ln\frac {E_xu(X(1))}{u(x)}d\nu(x)$
 and the infimum is taken over positive continuous functions on $M$.

Next, let $Y^{(i)}(n),\, i=2,...,\ell;\, n=0,1,2,...$ be i.i.d. $M$-valued
random variables with the distribution $\mu$, all of them independent of
the Markov chain $X(n),\, n\geq 0$. Then it is easy to see that 
\begin{equation}\label{3.18}
\lim_{N\to\infty}\frac 1N\ln E_x\exp\big(\sum_{n=1}^NW(X_n,Y^{(2)}(n),...,
Y^{(\ell)}(n))\big)=Q(W).
\end{equation}
Indeed, let $\cF_X$ be the $\sigma$-algebra generated by the Markov
chain $X(n),n\geq 0$. Then
\begin{eqnarray}\label{3.19}
&E_x\exp\big(\sum_{n=1}^NW(X(n),Y^{(2)}(n),...,Y^{(\ell)}(n))\big)\\
&=E_x\big(E_x(\exp(\sum_{n=1}^NW(X(n),Y^{(2)}(n),...,Y^{(\ell)}(n)))|
\cF_X)\big)\nonumber\\
&=E_x\exp(\sum_{n=1}^N\ln\hat W(X(n)))\nonumber
\end{eqnarray}
and (\ref{3.18}) follows. But now we have the standard situation for the
Markov chain $(X(n),Y^{(2)}(n),...,Y^{(\ell)}(n)),\, n\geq 0$, and so
by the Donsker--Varadhan variational formula (see \cite{DV1} and \cite{DV2}),
\begin{equation}\label{3.20}
Q(W)=\sup_{\nu\in\mathcal P(M\times\cdots\times M)}\big(\int 
W(x_1,x_2,...,x_\ell)d\nu(x_1,...,x_\ell)-I(\nu)\big)
\end{equation}
where 
\begin{eqnarray}\label{3.21}
&I(\nu)=-\inf_{u\in C_+(M\times\cdots\times M)}\int_{M\times\cdots\times M}\\
&\ln\frac {E_{x_1}\int u(X(1),x_1,...,x_\ell)d\mu(x_2)...d\mu(x_\ell)}
{u(x_1,...,x_\ell)}d\nu(x_1,...,x_\ell).\nonumber
\end{eqnarray}
It is known here (see, for instance, Proposition 5.1 in \cite{Ki1}) that there
 exists a unique $\nu=\nu_W$ on which the 
supremum in (\ref{3.20}) is attained and it follows from the standard theory
(see, for instance, \cite{Ki2}) that $I(\nu)$ is the rate functional for the second level
 large deviations both for the auxiliary occupational measures 
 \[
 \frac 1N\sum_{n=1}^N\delta_{(X_n,Y^{(2)}_n,...,Y^{(\ell)}_n)}
 \]
 and for our nonconventional occupational measures $\zeta_N$.\qed

\subsection{Continuous time case}\label{subsec3.3}

Similarly to the discrete time case, the main step in the proof of Theorem 
\ref{thm2.3} is to establish (\ref{2.19}) and to identify the limit there
as the spectral radius of the semigroup (\ref{2.20}).

From (\ref{2.16}) it follows that for any $t\geq t_0$ and every measurable
set $G\subset M$,
\begin{equation}\label{3.22}
P(t,x,G)=\int_Gp(t,x,y)d\nu(y)\,\,\,\mbox{with}\,\,\, C^{-1}\leq p(t,x,y)
\leq C.
\end{equation}
Furthermore, similarly to (\ref{3.5}) (see \cite{Do}), 
\begin{equation}\label{3.23}
(1-Ke^{-\ka t})p(y)\leq p(t,x,y)\leq (1+Ke^{-\ka t})p(y)
\end{equation}
where $p(y)=\frac {d\mu}{d\nu}(y)$ is the density of the unique invariant
measure $\mu$ of the Markov process $X$. Observe that (\ref{2.18}) implies
also that for any $j\geq k+1$ and $\gam>0$,
\begin{equation}\label{3.24}
\lim_{t\to\infty}(q_j(\gam t)-q_{j-1}(t))=\infty.
\end{equation}

Let $V=V(x_1,...,x_\ell)$ be a bounded measurable function on $M^\ell$ and for
$m=1,2,...,\ell$ set 
\begin{equation*}
V^{(m)}_{\mbox{cont}}(x_1,...,x_m)=\int_M...\int_MV(x_1,...,x_m,x_{m+1},...,
x_\ell)d\mu(x_{m+1})...d\mu(x_\ell)
\end{equation*}
with $V^{(\ell)}_{\mbox{cont}}=V$. Set $t_n(\gam,T)=\gam T+n(\gam+\gam^2)$ for
$n=0,1,2,...,M(\gam,T)-1$ where $M(\gam,T)=[(T(1-\gam )/(\gam+\gam^2)]$. Next,
for $a\leq b\leq c$ and $m=1,2,...,\ell$ we denote
\begin{eqnarray*}
&Z^{(m)}_x(a,b,c)=E_x\exp\big(\sum_{a\leq n<b}\int_{t_n(\gam,T)}^{t_n(\gam,T)
+\gam}V^{(m)}_{\mbox{cont}}(X(q_1(t)),...,X(q_m(t)))dt\\
&+\sum_{b\leq n<c}\int_{t_n(\gam,T)}^{t_n(\gam,T)+\gam}V^{(m-1)}_{\mbox{cont}}
(X(q_1(t)),...,X(q_{m-1}(t)))dt\big)
\end{eqnarray*}
and set $Z^{(m)}_x(a,b)=Z^{(m)}_x(a,b,b)$. Observe that 
$Z_x^{(\ell)}(0,M(\gam,T))$ does not contain the integration from 0 to $\gam T$
as well as the sum of integrals from $t_n(\gam,T)+\gam$ to $t_n(\gam,T)+\gam
+\gam^2$ which are both present in the integral from 0 to $T$, and so 
estimating these missing parts we arrive at the inequality
\begin{eqnarray}\label{3.25}
&\exp(-2C(V)\gam T)Z_x^{(\ell)}(0,M(\gam,T))\leq E_x\exp
\big(\int_0^TV(X(q_1(t)),\\
&...,X(q_\ell(t)))dt\big)\leq\exp(2C(V)\gam T)Z_x^{(\ell)}(0,M(\gam,T))\nonumber
\end{eqnarray}
where $C(V)=\sup_{(x_1,...,x_\ell)}|V(x_1,...,x_\ell)|$.

Denote by $\cF_t$ the $\sig$-algebra generated by $X(s),\, s\leq t$. Then
by (\ref{2.18}) and (\ref{3.24}) for all $T$ large enough if $n\geq 1,\,
T\geq t\geq t_n(\gam,T)$ and $s\leq t_{n-1}(\gam,T)+\gam$ then $X(q_1(s)),...,
X(q_m(s))$ and $X(q_1(t)),...,X(q_{m-1}(t))$ are $\cF_{q_m(t_n(\gam,T)
-\gam^2)}-$measurable. Hence,
\begin{eqnarray}\label{3.26}
&Z^{(m)}_x(0,n,M(\gam,T))=E_x\big( J_{m,n}\exp\big(\sum_{0\leq l<n}
\int_{t_l(\gam,T)}^{t_l(\gam,T)+\gam}\\
&V^{(m)}_{\mbox{cont}}(X(q_1(s)),...,X(q_m(s)))ds+\sum_{n+1\leq l<M(\gam,T)}
\int_{t_l(\gam,T)}^{t_l(\gam,T)+\gam}\nonumber\\
&V^{(m-1)}_{\mbox{cont}}(X(q_1(s)),...,X(q_{m-1}(s)))ds\big)\big)\nonumber
\end{eqnarray}
where 
\[
J_{m,n}=E_x\big(\exp\big(\int_{t_n(\gam,T)}^{t_n(\gam,T)+\gam}
V^{(m)}_{\mbox{cont}}(X(q_1(s)),...,X(q_m(s)))ds\big)\big\vert\cF_{q_m(t_n
(\gam,T)-\gam^2)}\big).
\]

Let
\[
\tilde J_{m,n}=\exp\big(\int_{t_n(\gam,T)}^{t_n(\gam,T)+\gam}E_x\big(
V^{(m)}_{\mbox{cont}}(X(q_1(s)),...,X(q_m(s)))\big\vert\cF_{q_m(t_n
(\gam,T)-\gam^2)}\big)ds\big).
\]
Since $|e^\al-1-\al|\leq\al^2$ if $|\al|\leq 1$ then 
\begin{equation}\label{3.27}
|J_{m,n}-\tilde J_{m,n}|\leq 2\gam^2(C(V))^2.
\end{equation}
On the other hand, by the Markov property
\begin{eqnarray}\label{3.28}
&\tilde J_{m,n}=\exp\big(\int_{t_n(\gam,T)}^{t_n(\gam,T)+\gam}ds\int_M
p\big(q_m(s)-q_m(t_n(\gam,T)-\gam^2),\\
&X(q_m(t_n(\gam,T)-\gam^2)),y\big)V^{(m)}_{\mbox{cont}}(X(q_1(s)),...
,X(q_{m-1}(s)),y)d\nu(y)\big).\nonumber
\end{eqnarray}
Set
\[
d_\gam(t)=\inf_{s\geq\gam t}\min_{k+1\leq j\leq\ell}\min(q_j(s)-
q_{j-1}(s\gam^{-1}),\, q_j(s)-q_j(s-\gam^2))
\]
and observe that $d_\gam(t)\to\infty$ as $t\to\infty$ for each fixed $\gam>0$
in view of the assumption (\ref{2.18}). Now, by (\ref{3.23}) and (\ref{3.28}),
\begin{eqnarray}\label{3.29}
&\exp(-Ke^{-\ka d_\gam(T)}C(V)\gam)\leq J_{m,n}\exp\big(-\int_{t_n(\gam,T)}^
{t_n(\gam,T)+\gam}V^{(m-1)}_{\mbox{cont}}(X(q_1(s)),\\
&...,X(q_{m-1}(s)))ds\big)\leq\exp(Ke^{-\ka d_\gam(T)}C(V)\gam).\nonumber
\end{eqnarray}

Employing (\ref{3.26})--(\ref{3.29}) for $n=M(\gam,T),\, M(\gam,T)-1,...,1$
 with each $m=\ell,\ell-1,...,k+1$ we obtain that
 \begin{equation}\label{3.30}
 \limsup_{T\to\infty}\frac 1T\big\vert\ln\big( Z_x^{(\ell)}(0,M(\gam,T))\big)
 -\ln\big( Z_x^{(k)}(0,M(\gam,T))\big)\big|=0.
 \end{equation}
 Now taking $\ln$ in (\ref{3.25}) and letting first $T\to\infty$ and then
 $\gam\to 0$ we obtain from (\ref{3.30}) and the definition of $Z_x^{(k)}$
 that
 \begin{eqnarray}\label{3.31}
&\lim_{T\to\infty}\frac 1T\big(\ln E_x\exp\big(\int_0^TV(X(q_1(t)),...,
X(q_\ell(t)))dt\big)\\
& -\ln E_x\exp\big(\int_0^TV^{(k)}_{\mbox{cont}}
(X(\al_1t),...,X(\al_kt))dt\big)\big)=0.\nonumber
\end{eqnarray}
If $k=1$ then $\frac 1T$ of the second expression in brackets in (\ref{3.31})
converges as $T\to\infty$ to the logarithm of the spectral radius of the
semigroup of operators $R^t_{\mbox{cont}}(V)$ defined in (\ref{2.20}). Thus,
the assertions of Theorems \ref{thm2.3} and \ref{thm2.4} follow from the well
known results on large deviations (see \cite{DV1}, \cite{DV2}, \cite{Ki2} and
\cite{DZ}) in the same way as in the discrete time case.   \qed 

\subsection{Nonconventional averaging}\label{subsec3.4}

According to \cite{Fr} the large deviations  estimates (\ref{2.34})
and (\ref{2.35}) follow once we establish (\ref{2.33}) for all continuous
functions $W_t(x_1,...,x_\ell)$ on $\bbR_+\times M^\ell$. First, we claim
that even without the assumption $k=1$,
\begin{eqnarray}\label{3.32}
&\lim_{\ve\to 0}\ve\big(\ln E_x\exp\big(\ve^{-1}\int_0^TW_t(X(q_1(t/\ve)),
...,X(q_\ell(t/\ve)))dt\big)\\
&-\ln E_x\exp\big(\ve^{-1}\int_0^TW_t^{(k)}(X(q_1(t/\ve)),...,X(q_k(t/\ve)))
dt\big)\big)=0\nonumber
\end{eqnarray}
where in the discrete time case $q_j$'s are extended to all $s\geq 0$ by writing
$q_j(s)=q_j([s])$ and we set
\[
W^{(k)}_t(x_1,...,x_k)=\ln\int_M...\int_M\exp(W_t(x_1,...,x_\ell))
d\mu(x_{k+1})...d\mu(x_\ell)
\]
while in the continuous time case we set
\[
W^{(k)}_t(x_1,...,x_k)=\int_M...\int_MW_t(x_1,...,x_\ell)
d\mu(x_{k+1})...d\mu(x_\ell).
\]
The proof of (\ref{3.32}) is the same as the proofs of (\ref{3.1}) in the
discrete time case and of (\ref{3.31}) in the continuous time case while
the dependence of $W_t$ on $t$ does not play any role in the arguments
employed there.

Next, when $k=1$ we arrive at the "conventional" setup and (\ref{2.33}) 
follows in the same way as in \cite{Fr} (see also \cite{Ki3}). \qed

\section{Large deviations for any $k\geq 1$: i.i.d. case}\label{sec4}
\setcounter{equation}{0}

Here we assume that $X(n),\, n\geq 1$ are i.i.d. random variables (vectors)
and rely on the decomposition (\ref{2.36}). In view of independency of
$S_{N,a}(V)$ for different $a\in A_N$ we can write
\begin{equation}\label{4.1}
Z_N(V)=E\exp\big(\sum^N_{n=1}V(X(n),X(2n),...,X(kn))\big)=\prod_{a\in A_N}
Z_{N,a}(V)
\end{equation}
where 
\[
Z_{\eta,a}(V)=E\exp\big(\sum_{b\in B_\eta(a)}V(X(b),X(2b),...,X(kb))\big)
\]
with $A_N$ and $B_\eta(a)$ defined in Section \ref{sec2}.

In order to study $Z_{N,a}(V)$ we introduce also
\[
B(a)=\{ b\geq 1:\, b=ar_1^{d_1}r_2^{d_2}\cdots r_m^{d_m}\,\,\mbox{for some
nonnegative integers}\,\, d_1,...,d_m\}.
\]
Observe that each $l=1,2,...,k$ can be
written uniquely in the form $l=r_1^{d_1(l)}r_2^{d_2(l)}\cdots r_m^{d_m(l)}$
for some nonnegative integers $d_1(l),...,d_m(l)$. Now, if $b=ar_1^{d_1}\cdots
r_m^{d_m}\in B(a)$ and $l=1,2,...,k$ then $lb=ar_1^{d_1+d_1(l)}\cdots
r_m^{d_m+d_m(l)}\in B(a)$. Next, consider the lattice $\bbZ^m$ and set
\[
\bbZ^m_+=\{ n=(n_1,...,n_m),\, n_i\geq 0\,\,\mbox{for all}\,\, i=1,...,m\}.
\]
Then the formula $\vf_a(n_1,...,n_m)=ar_1^{n_1}\cdots r_m^{n_m}$ provides
a one-to-one correspondence 
\[
\vf_a:\,\bbZ^m_+\rightarrow B(a)
\]
where, recall, $a$ is relatively prime with $r_1,...,r_m$. Set
\[
D(\rho)=\{ n=(n_1,...,n_m)\in\bbZ^m:\, n_i\geq 0,\, i=1,...,m\,\,\mbox{and}
\,\,\sum_{i=1}^mn_i\ln r_i\leq\rho\}.
\]
Then, clearly, 
\begin{equation}\label{4.2}
\vf_aD(\ln(N/a))=B_N(a).
\end{equation}
It follows that
\begin{equation}\label{4.3}
|B_N(a)|\leq\prod_{i=1}^m\big(1+\frac 1{\ln r_i}\ln\frac Na\big)\leq
\big(1+\frac 1{\ln 2}\ln\frac Na\big)^m
\end{equation}
where $|\Gam|$ denotes the cardinality of a set $\Gam$. Hence
\begin{equation}\label{4.4}
a\leq N2^{-(|B_N(a)|^{1/m}-1)}.
\end{equation}

Next, we claim that $Z_{N,a}(N)$ is determined only by $|B_N(a)|$ and not by
$N$ and $a$ themselves. Indeed, since $|D(\rho)|$ is nondecreasing in $\rho$
then it determines the set $D(\rho)$ itself, and so $|D(\ln(N/a))|=|B_N(a)|
=|B_{N/a}(1)|$ determines the set $B_{N/a}(1)$ in view of (\ref{4.2}). Set
$\hat {B}_\eta(a)=B_\eta(a)\cup\{ n:\,n=ln'$ for some $n'\in B_\eta(a)$ and 
$l=2,3,...,k\}$. Then we can write
\begin{equation*}
Z_{\eta,a}(V)=\int...\int\exp\big(\sum_{b\in B_\eta(a)}V(x_b,x_{2b},...,x_{kb})
\big)\prod_{b'\in\hat B_\eta(a)}d\mu(x_{b'}).
\end{equation*}
It is easy to see from here that $Z_{\eta,a}(V)=Z_{\eta/a,1}(V)$ for any
$\eta>0$ and an integer $a\geq 2$ relatively prime with $r_1,...,r_m$. 
Indeed, $Z_{\eta,a}(V)$ is determined by the labeled directed graph 
$\Gam_\eta(a)$ having $B_\eta(a)$ as its vertices and having arrows of
$k-1$ types so that an arrow with a label $l=2,3,...,k$ is drawn from 
$n\in B_\eta(a)$ to $n'\in B_\eta(a)$ if $n'=ln$. Clearly, the graphs
$\Gam_\eta(a)$ and $\Gam_{\eta/a}(1)$ are isomorphic in the sense that there
exists a one-to-one map $\vf:\, B_\eta(a)\to B_{\eta/a}(1)$ such that if
$n,n'\in B_\eta(a)$ and $n'=ln$ then $\vf n,\vf n'\in B_{\eta/a}(1)$ and
$\vf n'=l\vf n$. Since $X(n),\, n\geq 1$ are i.i.d., $Z_{\eta,a}(V)$ is 
determined, in fact, by the isomorphism class of $\Gam_\eta(a)$ and not
by  $\Gam_\eta(a)$ itself, and so $Z_{\eta,a}(V)=Z_{\eta/a,1}(V)$. Since
$|B_N(a)|$ determines the set $B_{N/a}(1)$ we conclude that it determines
$Z_{N,a}(V)$, as well, proving the claim.

Let $l=|B_N(a)|$ and set $R_l(V)=Z_{N,a}(V)$ since the latter depends only
on $l$ (and, of course, on $V$). Observe that
\begin{equation}\label{4.5}
\ln R_l(V)\leq lC(V)
\end{equation}
where $C(V)=\sup_{x_1,...,x_k\in M}|V(x_1,...,x_k)|$. Set $A^{(l)}_N=
\{ a\in A_N:\, |B_N(a)|=l\}$. By (\ref{4.4}),
\begin{equation}\label{4.6}
|A_N^{(l)}|\leq N2^{-(l^{1/m}-1)}.
\end{equation}
Observe that $|D(\rho)|$ is a nondecreasing right continuous piecewise
constant function and since $r_1,r_2,...,r_m$ are primes the jumps of
$|D_N(\rho)|$ can only be of size 1, i.e. for all $\tilde\rho>0$,
\[
|D(\tilde\rho)|-\lim_{\rho\uparrow\tilde\rho}|D(\rho)|\leq 1.
\]
It follows that 
\[
\rho_{\min}(l)=\inf\{\rho\geq 0:\, |D(\rho)|=l\}\,\,\mbox{and}\,\,
\rho_{\max}(l)=\sup\{\rho\geq 0:\, |D(\rho)|=l\}
\]
are well defined for each integer $l\geq 1$ and $\rho_{\min}(l)<\rho_{\max}(l)$.
Denote $\hat {A}_N^{(l)}=\{ a\in\bbN:\, Ne^{-\rho_{\max}(l)}\leq a\leq
Ne^{-\rho_{\min}(l)},\, a\,$ is relatively prime with $\,r_1,r_2,...,r_m\}$.
 Then by (\ref{4.2}) and the above,
\begin{equation}\label{4.7}
\frac 1N\big\vert |A_N^{(l)}|-|\hat {A}_N^{(l)}|\big\vert\leq\frac 1N\to 0
\,\,\mbox{as}\,\, N\to\infty.
\end{equation}

We will show next that the limit
\begin{equation}\label{4.8}
\lim_{N\to\infty}\frac 1N|\hat {A}_N^{(l)}|=(e^{-\rho_{\min}(l)}-
e^{-\rho_{\max}(l)})r
\end{equation}
exists with 
\begin{equation}\label{4.9}
r=1-\frac 12-\frac 13+\frac 1{2\cdot 3}-\frac 15+\frac 1{2\cdot 5}+
\frac 1{3\cdot 5}-\frac 1{2\cdot 3\cdot 5}+\cdots +(-1)^m\frac 1{r_1\cdot
r_2\cdots r_m}.
\end{equation}
Indeed, for each integer $n\geq 1$ set $G(n)=\{ in:\, i\in\bbZ_+\}$ and
$G^{(l)}_N(n)=\{ j\in G(n):\, Ne^{-\rho_{\max}(l)}\leq j\leq
Ne^{-\rho_{\min}(l)}\}$. Then (by the inclusion-exclusion principle),
\begin{equation}\label{4.10}
|\hat A_N^{(l)}|=|G_N^{(l)}(1)|-|G_N^{(l)}(2)|-|G_N^{(l)}(3)|+
|G_N^{(l)}(2\cdot 3)|+\cdots +(-1)^m|G_N^{(l)}(r_1\cdot r_2\cdots r_m)|.
\end{equation}
Since each $G(n)$ is an arithmetic progression with the difference $n$ we
obtain that
\begin{equation}\label{4.11}
\lim_{N\to\infty}\frac 1N|G_N^{(l)}(n)|=\frac 1n(e^{-\rho_{\min}(l)}-
e^{-\rho_{\max}(l)})
\end{equation}
and (\ref{4.8})--(\ref{4.9}) follows from (\ref{4.10})--(\ref{4.11}).

Observe that by (\ref{4.2}), (\ref{4.3}) and the definition of $\rho_{\min}$
and $\rho_{\max}$,
\begin{equation}\label{4.11+}
\rho_{\max}(l)\geq\rho_{\min}(l)\geq (l^{1/m}-1)\ln 2,
\end{equation}
and so we obtain from (\ref{4.1}) and (\ref{4.5})--(\ref{4.9}) that
\begin{eqnarray}\label{4.12}
&\frac 1N\ln Z_N(V)=\frac 1N\sum_{a\in A_N}\ln Z_{N,a}(V)\\
&=\frac 1N\sum_{1\leq l\leq (1+\frac 1{\ln 2}\ln\frac Na)^m}|A^{(l)}_N|
\ln R_l(V)\nonumber\\
&\longrightarrow r\sum_{l=1}^\infty(e^{-\rho_{\min}(l)}-
e^{-\rho_{\max}(l)})\ln R_l(V)\,\,\mbox{as}\,\, N\to\infty\nonumber
\end{eqnarray}
while the last series converges absolutely in view of (\ref{4.5}) and 
(\ref{4.11+}).
Furthermore, if $V=V_\la$ depends on a parameter $\la$ in a differentiable 
way with a derivative bounded by $\tilde C$ then each $\ln R_l(V_\la)$ is 
also differentiable in $\la$ with a derivative bounded by $\tilde Cl$. Hence,
in this case we can differentiate in $\la$ the
series in the right hand side of (\ref{4.12}) and the assertion of Theorem
\ref{thm2.6} follows. \qed 

\begin{remark}\label{rem4.1}
Arguments of the present section yield also moderate deviations estimates
for sums $S_N(V)$ given by (\ref{2.36}) in the above i.i.d. setup. Namely,
let $\bar V=\int V(x_1,x_2,...,x_k)d\mu(x_1)d\mu(x_2)\cdots d\mu(x_k)$,
where $\mu$ is the probability distribution of $X(1)$, and observe that 
$\bar V=EV(X(n),X(2n),...,X(kn))$ for any $n\geq 1$. Then for any $\ka\in
(0,\frac 12)$,
\begin{equation}\label{4.14}
\limsup_{N\to\infty}N^{2\ka-1}\ln P\{ N^{\ka-1}S_N(V-\bar V)\in K\}\leq
-\frac 12\La\inf_{u\in K}u^2
\end{equation}
for any closed set $K\subset\bbR$ and
\begin{equation}\label{4.15}
\liminf_{N\to\infty}N^{2\ka-1}\ln P\{ N^{\ka-1}S_N(V-\bar V)\in U\}\geq
-\frac 12\La\inf_{u\in U}u^2
\end{equation}
for any open set $U\subset\bbR$ provided that for any $\la\in\bbR$,
\begin{equation}\label{4.16}
\lim_{N\to\infty}N^{2\ka-1}\ln E\exp(\la N^{-\ka}S_N(V-\bar V))=\frac 12
\La^{-1}\la^2
\end{equation}
(cf. \cite{Fr} and \cite{DZ}). In order to compute the limit (\ref{4.16})
we observe relying on the same arguments as above that $\up_l(V)=
E(S_{N,a}(V-\bar V))^2$ depends only on $l=|B_N(a)|$ and on $V$ where, recall,
$S_{N,a}$ was defined in (\ref{2.36}). It follows that
\begin{equation}\label{4.17}
\ln Z_{N,a}(\la N^{-\ka}(V-\bar V))=\frac 12\la^2N^{-2\ka}\up_l(V)+
O(|\la|^3N^{-3\ka}\| V\|^3l^3)
\end{equation}
provided $|B_N(a)|=l$. Then in the same way as in (\ref{4.12}),
\begin{eqnarray}\label{4.18}
&\lim_{N\to\infty}N^{2\ka-1}\ln Z_N(\la N^{-\ka}(V-\bar V))\\
&=\frac 12\la^2\lim_{N\to\infty}N^{-1}\sum_{1\leq l\leq(1+\frac 1{\ln 2}
\ln\frac Na)^m}|A_N^{(l)}|\up_l(V)\nonumber\\
&=\frac 12\la^2r\sum_{l=1}^\infty
(e^{-\rho_{\min}(l)}-e^{-\rho_{\max}(l)})\up_l(V)\nonumber
\end{eqnarray}
and (\ref{4.16}) follows under a nondegeneracy condition $\up_l(V)\ne 0$
whenever $A_N^{(l)}\ne\emptyset$.
\end{remark}

\section{Nonconventional large deviations for dynamical systems}\label{sec5}
\setcounter{equation}{0}

In this section we discuss nonconventional large deviations results in the
dynamical systems case and 
a reader which is not familiar with hyperbolic dynamical systems and is
interested only in the probabilistic setup may skip this section altogether.
We assume now that $T:\, M\to M$ is either a subshift of finite
type or a $C^2$ expanding endomorphism or a hyperbolic diffeomorphism on
a compact Riemannian manifold (see \cite{Bo} and \cite{KH}). By the latter
we mean a $C^2$ Anosov diffeomorphism or, more generally, a $C^2$
diffeomorphism defined in a neighborhood of a hyperbolic attractor. We
identify now the probability space $(\Om,\cF,P)$ with $(M,\cB,\mu)$ where
$\cB$ is the Borel $\sig$-algebra on $M$ and $\mu$ is a Gibbs $T$-invariant
measure constructed by a H\" older continuous potential $g$ (see \cite{Bo}
and \cite{KH}). Let $k=1$ in (\ref{2.3}), (\ref{2.4}) and (\ref{2.17}),
(\ref{2.18}).

\begin{theorem}\label{thm5.1} Let $X(n)=X(n,\om)=X(n,x)=f(T^nx),\, n\geq 0$,
where $f$ is a H\" older continuous (vector) function, and we take also
 $q_j$'s as in Theorem \ref{thm2.1}. Let $k=1$ then for any 
$W_\la=W_\la(x_1,...,x_\ell)$ continuous in $x_1,...,x_\ell$,
\begin{eqnarray}\label{5.1}
&Q(W_\la)=\lim_{N\to\infty}\frac 1N\ln\int_M\exp\big(\sum_{n=1}^N
W_\la(T^{q_1(n)}x,...,T^{q_\ell(n)}x)\big)d\mu(x)\\
&=\gP (\ln\hat {W}_\la+g)\nonumber
\end{eqnarray}
with $\hat W$ defined by (\ref{2.5}), $g$ being the potential of $\mu$ and
$\gP (\cdot)$ being the topological pressure of a function in brackets for
the transformation $T$ (see \cite{Bo} and \cite{KH}). If the derivative
$dW_\la/d\la$ exists and is bounded in $x_1,...,x_\ell$ for each $\la$ then
$Q(W_\la)$ is differentiable in $\la$, as well. In the expanding and hyperbolic
cases the limit in (\ref{5.1}) remains the same if we integrate in (\ref{5.1})
either with respect to the normalized Riemannian volume or with respect to the
 Sinai-Ruelle-Bowen (SRB) measure $\mu=\SRB$ which is the Gibbs measure 
corresponding to the potential
$g=-\ln\vf$ where $\vf$ is the Jacobian of the differential $DT$ restricted
to unstable leaves (see \cite{Bo} and \cite{KH}). The large deviations estimates 
(\ref{2.9}) and (\ref{2.10}) hold true with the rate functional $J$ given by 
(\ref{2.8}) with $Q(W_\la)$ for $W_\la=\la F$ given by (\ref{5.1}) in place of
 $r(W_\la)$ in (\ref{2.8}).
  \end{theorem}
  
\begin{proof} For $T$ being a $C^2$ Axiom A diffeomorphism (in
particular, Anosov) in a neighborhood of an attractor or $T$ being
an expanding $C^2$ endomorphism of a Riemannian manifold $M$ (see
\cite{Bo}) let $\zeta$ be a finite Markov partition
for $T$. Then we can take $\cF_{kl}$ to be the finite $\sig$-algebra
generated by the partition $\cap_{i=k}^lT^i\zeta$. Another case for the above
theorem is when $T$ is a topologically mixing subshift of finite type, i.e.
 $T$ is the left shift on a subspace $\Xi$ of the space of one-sided
sequences $\vs=(\vs_i,i\geq 0), \vs_i=1,...,l_0$ such that $\vs\in\Xi$
if $\xi_{\vs_i\vs_{i+1}}=1$ for all $i\geq 0$ where $\Xi=(\xi_{ij})$ 
is an $l_0\times l_0$ matrix with $0$ and $1$ entries and such that $\Xi^n$ 
for some $n$
is a matrix with positive entries. Again, we take $\mu$ to be a Gibbs
invariant measure corresponding to some H\" older continuous function 
and to define $\cF_{kl}$ as the finite $\sig$-algebra generated by cylinder
sets with fixed coordinates having numbers from $k$ to $l$. The
exponentially fast $\psi$-mixing is well known in the above cases (see 
\cite{Bo}). In fact, convergence to zero of the modified $\psi-$mixing
coefficient $\psi_{P,\Pi}(n)$ holds true, as well, in the hyperbolic and
expanding case when $P$ is the normalized Riemannian volume and $\Pi=\SRB$.

If the function $W_\la=W_\la(x_1,...,x_\ell)$ is continuous
in $x_1,...,x_\ell$ then $\be_{W_\la}(n)$ from Proposition \ref{prop3.1}
tends to zero as $n\to\infty$, and so the condition (\ref{3.1}) will be 
satisfied here. It follows from \cite{Ki2} that
\[
\lim_{N\to\infty}\frac 1N\ln\int\exp\big(\sum_{n=1}^N\hat {W}_\la(f(T^nx))
\big)d\mu(x)=\gP (\ln\hat {W}_\la+g)
\]
and Theorem \ref{thm5.1} follows from Proposition 
\ref{prop3.1} and Corollary \ref{cor3.2} considered with $k=1$ since
in our circumstances  differentiability of the topological pressure in 
parameters of the potential is well known (see, for instance, \cite{Ru}
and \cite{PP}).
\end{proof}

\begin{remark}\label{rem5.2} (i) A version of Theorem \ref{thm2.2} can 
also be obtained in the present dynamical systems setup where the limit
$Q(W)=\gP (\ln\hat {W}(x)+g)$ is obtained in the same way as in Theorem 
\ref{thm5.1}. Since $\gP (q)$ is Gateaux differentiable at any H\" older
continuous $q$ (see \cite{Wa} and \cite{PP}) then $Q(W)$ is also Gateaux
differentiable at any H\" older continuous $W$ and the large deviations
for occupational measures 
\[
\zeta_N=\zeta_{N,x}=\frac 1N\sum_{n=1}^N\del_{\big(T^{q_1(n)}x,...,
T^{q_\ell(n)}x\big)}
\]
follow from Section 4.5.3 in \cite{DZ} with a rate function which is the
Fenchel--Legendre transform of $Q$.

(ii) Theorem \ref{2.6} provides a direct application to the dynamical systems
case when $T$ is a full shift (on a finite alphabet sequence space) considered
with a Bernoulli invariant measure taking $X(n)=f\circ T^n$ with
a function $f$ on the sequence space depending only on zero coordinate.
Nonconventional large deviations when $k>1$ for more general
cases (e.g. subshifts of finite type with Gibbs invariant measures, hyperbolic
and expanding transformations etc.) require more elaborate technique and they
will not be treated in this paper. 
\end{remark}


\begin{thebibliography}{Bow75}

\bibliography{matz_nonarticles,matz_articles}
\bibliographystyle{alpha}
\itemsep=\smallskipamount

\bibitem{As}
I. Assani, {\em Multiple recurrence and almost sure convergence for weakly
mixing dynamical systems}, Israel J. Math. {\bf 103}, 111--124 (1998).

\bibitem{Be}
V. Bergelson, {\em Weakly mixing PET}, 
Ergod. Th.\& Dynam. Sys. {\bf 7},
 337--349 (1987).


\bibitem{Bo}
R. Bowen, {\em Equilibrium States and the Ergodic Theory of Anosov 
Diffeomorphisms}, Lecture Notes in Math. 470, Springer--Verlag, 2nd ed., 2008.


 
 \bibitem{BLM}
 V. Bergelson, A. Leibman and  C.G. Moreira, {\em From discrete-to
  continuous time ergodic theorems}, Ergod. Th.\& Dyn. Sys., to appear.
  
  
\bibitem{Co}
G. Contreras, {\em Regularity of topological and metric entropy of
hyperbolic flows}, Math. Z. 210 (1992), 97--111.  
  
  
\bibitem{CCGR}  
G. Carinci, J-R Chazottes, C. Giardina and F. Redig, {\em Nonconventional
averages along arithmetic progressions and lattice spin systems}, Indag. Math.
23 (2012), 589--602.  
  

\bibitem{Dol} D. Dolgopyat, {\em Limit theorems for partially hyperbolic 
systems}, Trans. Amer. Math. Soc. 356 (2003), 1637--1689.


\bibitem{Do} J. Doob, {\em Stochastic Processes}, Wiley, New York, 1953.

\bibitem{DV1} M.D. Donsker and S.R.S. Varadhan, {\em Asymptotic evaluation
of certain Markov processes expectations for large time. I,}
Comm. Pure Appl. Math. 28 (1975), 1--47.


\bibitem{DV2} M.D. Donsker and S.R.S. Varadhan, {\em On the variational 
formula for the principal eigenvalue for operators with maximum principle},
Proc. Nat. Acad. Sci. U.S.A. 72 (1975), 780--783.


\bibitem{DZ} A. Dembo and O. Zeitouni, {\em Large Deviations Techniques
and Applications, 2nd. ed.}, Springer, Heidelberg, 1998.


\bibitem{Fr} M.I. Freidlin, {\em The averaging principle and theorems
on large deviations,} Russ. Math. Surv., {\bf 33,} No.5 (1978), 107--160.



\bibitem{Fu}
H. Furstenberg, {\em Nonconventional ergodic averages}, Proc.
Symp. Pure Math. 50, 43--56 (1990).





\bibitem{IL}
 I.A. Ibragimov and Yu.V. Linnik, {\em Independent and Stationary Sequences
 of Random Variables}, Wolters--Noordhoff, Groningen (1971).


\bibitem{Ki1} Yu. Kifer, {\em Principal eigenvalues, topological pressure,
and stochastic stability of equilibrium states,} Israel J. Math. 70 (1990),
1--47.


\bibitem{Ki2} Yu. Kifer, {\em Large deviations in dynamical systems
 and stochastic processes,} Trans. Amer. Math. Soc.,  321 (1990),
 505--524.
 

 
 \bibitem{Ki3} Yu. Kifer, {\em Averaging in dynamical systems and
large deviations,} Invent. Math., 110 (1992), 337--370. 


\bibitem{Ki4-} Yu. Kifer, {\em Nonconventional limit theorems}, Prob.
Th. Rel. Fields 148 (2010), 71--106.


\bibitem{Ki4} 
Yu. Kifer, {\em A nonconventional strong law of large numbers and fractal
dimensions of some multiple recurrence sets}, Stoch. Dynam. 12 (2012),
1150023.


\bibitem{Kr}
M.A. Krasnoselskii, {\em Positive Solutions of Operator Equations}, Noordhoff,
Groningen, 1964.


\bibitem{KH} A. Katok and B. Hasselblatt, {\em Introduction to the
Modern Theory of Dynamical Systems,} (1995), Cambridge Univ. Press, 
Cambridge.



\bibitem{KV}
Yu. Kifer and S.R.S. Varadhan, {\em Nonconventional limit theorems in
discrete and continuous time via martingales}, Ann. Probab., to appear.


\bibitem{PP} 
W. Parry and M. Pollicott, {\em Zeta Functions and the Periodic Structure
of Hyperbolic Dynamics}, Ast\' erisque 187-188, Soc. Math. de France,
1990.


\bibitem{Ru}
D. Ruelle, {\em Thermodynamic Formalism}, Addison-Wesley, Reading, 1978.
 


\bibitem{SVM} J.A. Sanders, F.Verhurst and J. Murdock {\em Averaging Methods 
in Nonlinear Dynamical Systems,} 2nd ed. (2007), Springer, New York.


\bibitem{Wa} P. Walters, {\em Differentiability properties of the pressure
of a continuous transformation on a compact metric space}, J. London Math.
Soc. (2) 46 (1992), 471--481.


\end{thebibliography}
\end{document}